\documentclass[11pt,times,twoside]{article}

\usepackage{graphicx,latexsym,euscript,makeidx,color}
\usepackage{amsmath,amsfonts,amssymb,amsthm,mathrsfs}
\usepackage{enumerate}

\usepackage[colorlinks,linkcolor=blue,anchorcolor=green,citecolor=red]{hyperref}

\usepackage{geometry}
\def\5n{\negthinspace \negthinspace \negthinspace \negthinspace \negthinspace }
\def\4n{\negthinspace \negthinspace \negthinspace \negthinspace }
\def\3n{\negthinspace \negthinspace \negthinspace }
\def\2n{\negthinspace \negthinspace }
\def\1n{\negthinspace }
   \def\cA{{\cal A}}  
   \def\cB{{\cal B}}  
   \def\cC{{\cal C}}  
   \def\cD{{\cal D}}  
\def\dbE{\mathbb{E}}     
\def\dbF{\mathbb{F}}   \def\cF{{\cal F}}  
   \def\cG{{\cal G}}

   \def\cJ{{\cal J}}

\def\dbP{\mathbb{P}}   \def\cP{{\cal P}}  
   \def\cQ{{\cal Q}}  
\def\dbR{\mathbb{R}}   \def\cR{{\cal R}}  
\def\dbS{\mathbb{S}}   \def\cS{{\cal S}}  
     
   \def\cU{{\cal U}}  
   \def\cV{{\cal V}}  
     
   \def\cX{{\cal X}}

\def\Om{\Omega}

\def\ss{\smallskip}                \def\lt{\left}
\def\ms{\medskip}                \def\rt{\right}
               \def\lan{\langle}
\def\nn{\nonumber}               \def\ran{\rangle}
\def\ra{\rightarrow}                         \def\ds{\displaystyle}             \def\ns{\noalign{\ss}}
\def\no{\noindent}        \def\q{\quad}                      \def\llan{\left\langle}
\def\ts{\times}           \def\qq{\qquad}                    \def\rran{\right\rangle}

         \def\rf{\eqref}                    \def\Blan{\Big\lan\!}
  \def\deq{\triangleq}               \def\Bran{\!\Big\ran}
            \def\({\Big (}
\def\les{\leqslant}                  \def\){\Big )}
\def\ges{\geqslant}          \def\[{\Big[}
           \def\]{\Big]}

                            \def\hp{\hphantom}

\def\a{\alpha}                 
\def\b{\beta}            \def\d{\delta}        
             \def\Si{\Sigma}  
\def\e{\varepsilon}             
         \def\f{\varphi}  \def\i{\infty}   

\def\bde{\begin{definition}\label}    \def\ede{\end{definition}}
\def\be{\begin{equation}}
\def\bel{\begin{equation}\label}      \def\ee{\end{equation}}
\def\bt{\begin{theorem}\label}        \def\et{\end{theorem}}
\def\bc{\begin{corollary}\label}      \def\ec{\end{corollary}}
\def\bl{\begin{lemma}\label}          \def\el{\end{lemma}}
\def\bp{\begin{proposition}\label}    \def\ep{\end{proposition}}
\def\bas{\begin{assumption}\label}    \def\eas{\end{assumption}}
\def\br{\begin{remark}\label}         \def\er{\end{remark}}
\def\bex{\begin{example}\label}       \def\ex{\end{example}}
\def\ba{\begin{aligned}}              \def\ea{\end{aligned}}
\def\ben{\begin{enumerate}}           \def\een{\end{enumerate}}

\newtheorem{theorem}{Theorem}[section]
\newtheorem{definition}[theorem]{Definition}
\newtheorem{proposition}[theorem]{Proposition}
\newtheorem{corollary}[theorem]{Corollary}
\newtheorem{lemma}[theorem]{Lemma}
\newtheorem{remark}[theorem]{Remark}
\newtheorem{example}[theorem]{Example}

\makeatletter
   
   \@addtoreset{equation}{section}
\makeatother

\begin{document}

\title{{\bf General Indefinite Backward Stochastic Linear-Quadratic Optimal Control Problems}
}

\author{Jingrui Sun\thanks{Department of Mathematics, Southern University of Science and Technology, Shenzhen, Guangdong, 518055, China
(Email: {\tt sunjr@sustech.edu.cn}). This author is supported by NSFC grant 11901280 and Guangdong Basic
and Applied Basic Research Foundation 2021A1515010031.}~,~~~
Jiaqiang Wen\thanks{Department of Mathematics, Southern University of Science and Technology, Shenzhen, Guangdong, 518055, China (Email: {\tt wenjq@sustech.edu.cn}).
This author is supported by National Natural Science Foundation of China (Grant No. 12101291), Natural Science Foundation of Guangdong Province of China (Grant No. 2214050003543), and SUSTech start-up fund (Grant No. Y01286233).}~,~~~
Jie Xiong\thanks{Department of Mathematics and SUSTech International center for Mathematics, Southern University of Science and Technology, Shenzhen, Guangdong, 518055, China (Email: {\tt xiongj@sustech.edu.cn}).
This author is supported partially by National Natural Science Foundation of China (Grants No. 61873325 and 11831010), and SUSTech start-up fund (Grant No. Y01286120).}}

\maketitle


\no\bf Abstract. \rm
A general backward stochastic linear-quadratic optimal control problem is studied, in which both the state equation and the cost functional contain the nonhomogeneous terms. The main feature of the problem is that the weighting matrices in the cost functional are allowed to be indefinite and cross-product terms in the control and the state processes are present. Necessary and sufficient conditions for the solvability of the problem are obtained, and a characterization of the optimal control in terms of forward-backward stochastic differential equations is derived. By a Riccati equation approach, a general procedure for constructing optimal controls is developed and the value function is obtained explicitly.

\ms

\no\bf Key words: \rm Backward stochastic differential equation, linear-quadratic, optimal control, Riccati equation, nonhomogeneous term.

\ms

\no\bf AMS subject classifications. \rm
93E20, 49N10, 49N35.

\section{Introduction}

Due to its wide range of applications, the theory of stochastic optimal control developed
rapidly in the past few decades.
The stochastic control is a natural and effective method that can solve the uncertainty,
interference, and ambiguity emerging in real-world control problems.
As an important class of optimal control problems, the forward stochastic linear-quadratic problem
has been studied by a lot of researchers (see Wonham \cite{Wonham1968} and Davis \cite{Davis1977},
and the references cited therein).

\ms

In the historical development of stochastic optimal control, the backward stochastic differential equation
(BSDE, for short) plays a central role, which was introduced by Bismut \cite{Bismut1973} for the linear case
and by Pardoux--Peng \cite{Pardoux-Peng1990} for the nonlinear situation.
Linear BSDEs serve as the adjoint equation of the state equation in the study of the maximum principle
of stochastic optimal control problems (see Bismut \cite{Bismut1973} and Yong--Zhou \cite{Yong-Zhou1999}).
Control problems of BSDEs are also attractive and important, not only due to the theoretical level,
but also their applications in finance; see, for example, Ma--Yong \cite{Ma-Yong1999}, Pham \cite {Pham2009},
Peng \cite{Peng2010}, Zhang \cite{Zhang2017}, and the references cited therein.

\ms

In this paper, we study a class of quadratic control problems for linear BSDEs with nonhomogeneous terms,
in which the weighting matrices in the cost functional are allowed to be indefinite and cross-product terms
in the control and the state processes are present.
To precisely state our problem, let $(\Omega,\cF,\dbF,\dbP)$ be a complete filtered probability space on which
a one-dimensional standard Brownian motion $W=\{W(t);t\ges0\}$ is defined, where $\dbF=\{\cF_t\}_{t\ges0}$
is usual augmentation of the natural filtration generated by $W$.
For a random variable $\xi$, we write $\xi\in\cF_t$ if $\xi$ is $\cF_t$-measurable;
and for a stochastic process $\f$, we write $\f\in\dbF$ if it is progressively measurable with respect to $\dbF$.
Consider the following controlled linear BSDE on a finite horizon $[0,T]$:
\bel{state}\left\{\ba
\ds dY(t) &= [A(t)Y(t)+B(t)u(t)+C(t)Z(t)+f(t)]dt+Z(t)dW(t),  \\
\ns\ds Y(T) &= \xi,
\ea\right.\ee
where the coefficients $A,C:[0,T]\ra\dbR^{n\ts n}$ and $B:[0,T]\ra\dbR^{n\ts m}$ of the {\it state equation}
\rf{state} are given bounded deterministic functions; the nonhomogeneous term $f:[0,T]\ts\Om\ra\dbR^{n}$ is an $\dbF$-progressively
measurable process; and the terminal value $\xi$ belongs to the space
$$L^2_{\cF_T}(\Om;\dbR^n)\deq\Big\{\xi:\Om\to\dbR^n\bigm|\xi\in\cF_T  \text{ and }\dbE|\xi|^2<\i\Big\}.$$
The control process $u$, valued in $\dbR^m$, is taken from
\begin{align*}
  \cU\deq \bigg\{u:[0,T]\ts\Om\rightarrow\dbR^{m}~\big|&~u\in\dbF \text{ and } \dbE\int_0^T|u(s)|^2ds<\i \bigg\}.
\end{align*}
The criterion for the performance of $u$ is given by the following quadratic functional
\begin{equation}\label{cost}
\begin{aligned}
\ds J\big(\xi;u\big) &\deq\dbE\bigg\{\lan GY(0),Y(0)\ran+2\lan g,Y(0)\ran \\
\ns\ds &\hp{=\ } +\int_0^T\bigg[
\Blan\begin{pmatrix}Q(t)&S_1^\top(t)&S_2^\top(t)\\S_1(t)&R_{11}(t)&R_{12}(t)\\S_2(t)&R_{21}(t)&R_{22}(t) \end{pmatrix}\!
\begin{pmatrix}Y(t)\\Z(t)\\u(t)\end{pmatrix}\!,
\begin{pmatrix}Y(t)\\Z(t)\\u(t)\end{pmatrix}\Bran \\
\ns\ds &\hp{=\ } +2\Blan\begin{pmatrix}q(t)\\ \rho_1(t)\\ \rho_2(t) \end{pmatrix}\!,
   \begin{pmatrix}Y(t)\\Z(t)\\u(t)\end{pmatrix}\Bran\bigg] dt\bigg\},
\end{aligned}
\end{equation}
where the superscript $\top$ denotes the transpose of a matrix; $G$ is a symmetric $n\ts n$ constant
matrix; $g$ is a constant, $q$, $\rho_1$ and $\rho_2$ are $\dbF$-progressively measurable processes; and
\begin{align*}
  Q,\quad S=\begin{pmatrix}S_1\\ S_2\end{pmatrix},\quad  R=\begin{pmatrix}R_{11}&R_{12}\\R_{21}&R_{22}\end{pmatrix}
\end{align*}
are bounded deterministic matrix-valued functions of proper dimensions over $[0, T]$ such
that the blocked matrix in the cost functional is symmetric.
The backward stochastic optimal control problem of interest is as follows.

\ms

\no{\bf Problem (BSLQ).} For a given terminal state $\xi\in L^2_{\cF_T}(\Om;\dbR^n)$, find a control $u^*\in\cU$ such that
\begin{align}\label{Value}
  J(\xi;u^*)=\inf_{u\in\cU}J(\xi;u)\equiv V(\xi).
\end{align}
Due to the linearity of the backward state equation \rf{state} and the quadratic form of the cost \rf{cost},
we call the above problem a {\it backward stochastic linear-quadratic (LQ) optimal control problem} (BSLQ problem).
A process $u^*\in\cU$ satisfying \rf{Value} is called an {\it optimal control} of Problem (BSLQ) for the terminal state $\xi$,
and the corresponding adapted solution $(Y^*, Z^*)$ of the state equation \rf{state} is called an {\it optimal state process}.
The function $V$ is called the {\it value function} of Problem (BSLQ).
When the coefficients $f,g,q,\rho_1,\rho_2$ vanish, we denote the corresponding Problem (BSLQ) by Problem (BSLQ)$^0$.
The corresponding cost functional and value function are denoted by $J^0(\xi;u)$ and $V^0(\xi)$, respectively.

\ms

\ms

The BSLQ optimal control problem without nonhomogeneous terms was first studied by Lim--Zhou \cite{Lim-Zhou2001},
where all the weighting matrices are assumed to be positive semidefinite and the quadratic cost functional is
independent of the cross terms of $(Y,Z,u)$.
Applying a forward formulation and a limiting procedure, together with the completion-of-squares technique,
a complete solution for such a BSLQ optimal control problem was obtained in \cite{Lim-Zhou2001}.
A couple of follow-up works have appeared afterward; see, for instance, Huang--Wang--Xiong \cite{Huang-Wang-Xiong2009}
and Wang--Wu--Xiong \cite{Wang-Wu-Xiong2012} considered BSLQ optimal control problem with partial information;
Huang--Wang--Wu \cite{Huang-Wang-Wu2016} investigated a backward mean-field linear-quadratic-Gaussian game with
full and partial information;
Wang--Xiao--Xiong \cite{Wang-Xiao-Xiong2018} studied BSLQ optimal control problem with asymmetric information;
a dynamic game of linear BSDE systems with mean-field interactions was studied in Du--Huang--Wu \cite{Du-Huang-Wu2018};
a thorough investigation on BSLQ optimal control problem with random coefficients was further carried out
in Sun--Wang \cite{Sun-Wang2019};
a general mean-field BSLQ optimal control problem was investigated in Li--Sun--Xiong \cite{Li-Sun-Xiong2019};
and based on \cite{Li-Sun-Xiong2019,Lim-Zhou2001}, a theory of optimal control for controllable stochastic linear
systems was developed in Bi--Sun--Xiong \cite{Bi-Sun-Xiong2020}.
It is worth pointing out that, the key point of the above-mentioned works is that they assume the positive/nonnegative
definiteness condition imposed on the weighting matrices, and most of their cost functionals are independent
of the cross terms in $(Y,Z,u)$ and nonhomogeneous terms are not present.

\ms

We say that the stochastic LQ optimal control is {\it indefinite}, if the weighting matrices in the cost functional $J(\xi,u)$,
are not necessarily positive semi-definite.
Not assuming the positive definiteness/ semi-definiteness on the weighting matrices will bring great challenge
for solving Problem (BSLQ).
Recently, Sun--Wu--Xiong \cite{Sun-Wu-Xiong2021} considered a homogeneous backward stochastic LQ optimal control problem
and obtained the optimal control for the indefinite case.
However, their model is not general enough due to the lack of homogeneous terms.
For this reason, their results cannot directly apply to solving some related problems, especially the two-person
zero-sum Stackelberg game.
In this paper, we study a general indefinite BSLQ optimal control problem, in which both the state equation and
the cost functional contain nonhomogeneous terms.
As we shall see in Section 4, the nonhomogeneous terms bring lots of difficulties when constructing the optimal
control of Problem (BSLQ).
For example, we need to reconstruct the representation of the solution $Z$ and the optimal control $u$ in terms
of $X$, the solution of the corresponding forward dual process.
We shall first derive necessary and sufficient conditions for the existence of optimal controls,
and then characterize the optimal control by means of forward-backward stochastic differential equations
(FBSDEs, for short).
Finally, with this characterization, we develop a general procedure for constructing the optimal control
and the value function of Problem (BSLQ).

\ms

The rest of the paper is structured as follows.
We give the preliminaries and collect some recently developed results on general forward stochastic LQ optimal
control problems in Section 2. Characterization of the optimal control is presented in Section 3,
by means of FBSDEs.
In Section 4, we first simplify Problem (BSLQ) and construct the optimal control in the case that the cost functional
is uniformly convex, and then present the general results.
Section 5 concludes the paper.

\section{Preliminaries}

First, we introduce some notation. Let $\mathbb{R}^{n \times m}$ be the Euclidean space of $n \times m$ real matrices, equipped with the Frobenius inner product
$$
\lan M, N\ran=\operatorname{tr}(M^{\top} N), \quad M, N \in \mathbb{R}^{n \times m},
$$
where $\operatorname{tr}\left(M^{\top} N\right)$ is the trace of $M^{\top} N$. The norm induced by the Frobenius inner product is denoted by $|\cdot| .$ The identity matrix of size $n$ is denoted by $I_{n} .$ When no confusion arises, we often suppress the index $n$ and write $I$ instead of $I_{n} .$ Let $\mathbb{S}^{n}$ be the subspace of $\mathbb{R}^{n \times n}$ consisting of symmetric matrices. For $\mathbb{S}^{n}$ -valued functions $M$ and $N$, we write $M \geqslant N$ (respectively, $M>N)$ if $M-N$ is positive semidefinite (respectively, positive definite) almost everywhere (with respect to the Lebesgue measure), and write $M \gg 0$ if there exists a constant $\delta>0$ such that $M \geqslant \delta I_{n}$. For a subset $\mathbb{H}$ of $\mathbb{R}^{n \times m}$, we denote by $C([0, T] ; \mathbb{H})$ the space of continuous functions from $[0, T]$ into $\mathbb{H}$, and by $L^{\infty}(0, T ; \mathbb{H})$ the space of Lebesgue measurable, essentially bounded functions from $[0, T]$ into $\mathbb{H}$. Besides the space $L_{\mathcal{F}_{t}}^{2}\left(\Omega ; \mathbb{R}^{n}\right)$ introduced previously, the following spaces of stochastic processes will also be frequently used in the sequel:
\begin{align*}
L_{\mathbb{F}}^{2}(0,T;\mathbb{H})=&\ \bigg\{\varphi:[0, T] \times \Omega \rightarrow \mathbb{H} \mid \varphi \text { is $\dbF$-progressively measurable and } \nn \\
&\qq\ \dbE  \int_{0}^{T}|\varphi(t)|^{2} d t<\infty\bigg\}, \nn \\
L_{\dbF}^{2}(\Omega ; C([0, T] ; \mathbb{H}))=&\ \bigg\{\varphi:[0, T] \times \Omega \rightarrow \mathbb{H} \mid \varphi \text { has continuous paths, $\dbF$-adapted and }\nn  \\
&\qq \left. \dbE \left[\sup _{0 \leqslant t \leqslant T}|\varphi(t)|^{2}\right]<\infty\right\} ,\nn \\
L_{\dbF}^{2}(\Omega ; L^1([0, T] ; \mathbb{H}))=&\ \bigg\{\varphi:[0, T] \times \Omega \rightarrow \mathbb{H} \mid \varphi \text { is $\dbF$-progressively measurable and } \nn \\
&\qq \left. \dbE \left[\sup _{0 \leqslant t \leqslant T}|\varphi(t)|^{2}\right]<\infty\right\}. \nn
\end{align*}
For the coefficients of the state equation \rf{state} and the weighting matrices of the cost functional \rf{cost}, we impose the following conditions.
\begin{itemize}
  \item [\bf{(A1)}] The coefficients of the state equation \rf{state} satisfy
      $$A\in L^{\i}(0,T;\dbR^{n\ts n}),\ B\in L^{\i}(0,T;\dbR^{n\ts m}),\
       C\in L^{\i}(0,T;\dbR^{n\ts n}),\ f\in L^2_{\dbF}(0,T;\dbR^n).$$
  \item [\bf{(A2)}]  The coefficients among the nonhomogeneous term and the weighting matrices in the cost functional \rf{cost} satisfy
\begin{align*}
\ds       &G\in\dbS^n,\ Q\in L^{\i}(0,T;\dbS^n),\ S\in L^{\i}(0,T;\dbR^{(n+m)\ts n}),\
      R\in L^{\i}(0,T;\dbS^{n+m}), \nn \\
\ns\ds     & g\in \dbR^n,\ q\in L^2_{\dbF}(\Om;L^1(0,T;\dbR^n)),\ \rho_1\in L^2_{\dbF}(\Om;L^1(0,T;\dbR^n)),\
      \rho_2\in  L^2_{\dbF}(0,T;\dbR^m). \nn
\end{align*}
\end{itemize}
We present the following lemma concerning the well-posedness of the state equation \rf{state},
which is a direct consequence of the theory of linear BSDEs (see Chapter 7 of Yong--Zhou \cite{Yong-Zhou1999}).

\begin{lemma}\label{estimate}
Under the assumption (A1), for any $(\xi,u)\in L^2_{\cF_T}(\Om;\dbR^n)\ts \cU$, the state equation \rf{state} admits a unique adapted solution
\begin{align*}
  (Y,Z)\in L^2_{\dbF}(\Om;C([0,T];\dbR^n))\ts L^2_{\dbF}(0,T;\dbR^n).
\end{align*}
Furthermore, there exists a constant $K>0$, independent of $\xi$ and $u$, such that
\begin{align*}
\dbE\bigg[\sup_{0\les t\les T}|Y(t)|^2+\int_0^T|Z(t)|^2dt\bigg]\les
K\dbE\bigg[|\xi|^2+\int_0^T|u(t)|^2dt+\int_0^T|f(t)|^2dt\bigg].
\end{align*}
\end{lemma}
We next collect some results from forward stochastic LQ optimal control theory, which will be used to constructing the optimal control of Problem (BSLQ). Consider the forward linear stochastic differential equation on a finite time horizon $[0,T]$:
\bel{SDE}\left\{\ba
\ds d\cX(t) &=\big[\cA(t)\cX(t)+\cB(t)v(t)+b(t)\big]dt
+\big[\cC(t)X(t)+\cD(t)v(t)+\sigma(t)\big]dW(t), \\
\ns\ds \cX(0) &=x,
\ea\right.\ee
and the following cost functional
\begin{align}\label{cost-SDE}
\ds \cJ\big(x;v\big)\deq&~\dbE\Bigg\{\big\lan \cG\cX(T),\cX(T)\big\ran
+2\big\lan \tilde{g},\cX(T)\big\ran \nonumber \\
\ns\ds &~ +\int_0^T\Bigg[
\llan\begin{pmatrix}\cQ(t)&\cS^\top(t)\\\cS(t)&\cR(t) \end{pmatrix}
\begin{pmatrix}\cX(t)\\v(t)\end{pmatrix},
\begin{pmatrix}\cX(t)\\v(t)\end{pmatrix}\rran \\
\ns\ds &~ +2\llan\begin{pmatrix}\tilde{q}(t)\\ \tilde{\rho}(t)\end{pmatrix},
   \begin{pmatrix}\cX(t)\\ v(t)\end{pmatrix}\rran\Bigg] ds\Bigg\}, \nonumber
\end{align}
where in the equations \rf{SDE} and \rf{cost-SDE},
\begin{align*}
\ds &\cA,\cC\in L^{\i}(0,T;\dbR^{n\ts n}),\ \cB,\cD\in L^{\i}(0,T;\dbR^{n\ts m}),\
  b,\sigma\in L^2_{\dbF}(0,T;\dbR^n), \nn \\
\ns\ds &\cG\in\dbS^n,\ \cQ\in L^{\i}(0,T;\dbS^n),\ \cS\in L^{\i}(0,T;\dbR^{m\ts n}),\
 \cR\in L^{\i}(0,T;\dbS^{m}),\nn \\
\ns\ds & \tilde{g}\in \dbR^n,\ \tilde{q}\in L^2_{\dbF}(\Om;L^1(0,T;\dbR^n)),\ \tilde{\rho}\in L^2_{\dbF}(\Om;L^1(0,T;\dbR^m)).\nn
\end{align*}
The stochastic linear-quadratic optimal control problem of the forward type is stated as follows.

\ms

{\bf Problem (FSLQ).} For a given initial state $x\in\dbR^n$, find a control $v^*\in\cU$ such that
\begin{align}\label{FValue}
  \cJ(x;v^*)=\inf_{u\in\cU}\cJ(x;v)\equiv \cV(x).
\end{align}
The control $v^{*}$ (if it exists) in \rf{FValue} is called an open-loop optimal control for the initial state $x$, and $\mathcal{V}(x)$ is called the value of Problem (FSLQ) at $x$. Note that Problem (FSLQ) is an indefinite LQ optimal control problem, since we do not require the weighting matrices to be positive semidefinite.
The following propositions establish the solvability of Problem (FSLQ) under a condition that is nearly necessary for the existence of open-loop optimal controls, and establish a regularity of the solution to the Riccati equation. We refer the reader to Sun--Yong \cite{Sun-Yong 2014}, Sun--Li--Yong \cite{Sun-Li-Yong 2016} and the recent book \cite{Sun-Yong 2020}
for proofs and further information.
\begin{proposition}
Assume that there exists a constant $\a>0$ such that
\begin{equation}\label{2.1}
\cJ(0;u) \ges \a\|u\|^2,\qq\forall u\in\cU.
\end{equation}
Then the following Riccati differential equation
\bel{Riccati-SDE}\left\{\ba
\ds & \dot{\cP}+\cP\cA+\cA^{\top}\cP+\cC^{\top}\cP\cC+\cQ  \\
\ns\ds &\hp{\dot{\cP}} -(\cP\cB+\cC^{\top}\cP\cD+\cS^{\top})(\cR+\cD^{\top}\cP\cD)^{-1}
(\cB^{\top}\cP+\cD^{\top}\cP\cC+\cS)=0,\\
\ns\ds  & \cP(T)=\cG,
\ea\right.\ee
admits a unique solution $\cP\in C([0,T];\dbS^n)$ such that
$$\cR+\cD^{\top}\cP\cD\gg0.$$
In addition, for each initial state $x$, a unique open-loop optimal control exists and is given
by the following closed-loop form:
\begin{align*}
\ds v^*(t)=& -(\cR+\cD^{\top}\cP\cD)^{-1}(\cB^{\top}\cP+\cD^{\top}\cP\cC+\cS)X^*\\
\ns\ds        & -(\cR+\cD^{\top}\cP\cD)^{-1}(\cB^\top\eta
       +\cD^{\top}\zeta+\cD^{\top}\cP\tilde{\sigma}+\tilde{\rho}),\qq t\in[0,T],
\end{align*}
where $X^*$ is the solution of the following closed-loop system:
$$\left\{\ba
\ds dX^*(t) &=\Big\{\big[\cA-\cB(\cR+\cD^\top \cP\cD)^{-1}(\cB^\top \cP+\cD^\top \cP\cC+\cS)\big]X^*\\
\ns\ds &\hp{=\ } -\cB(\cR+\cD^{\top}\cP\cD)^{-1}
(\cB^\top\eta+\cD^{\top}\zeta+\cD^{\top}\cP\sigma+\tilde{\rho})+b\Big\}dt\\
\ns\ds &\hp{=\ } +\Big\{\big[\cC-\cD(\cR+\cD^\top \cP\cD)^{-1}(\cB^\top \cP
+\cD^\top \cP\cC+\cS)\big]X^*\\
\ns\ds &\hp{=\ } -\cD(\cR+\cD^{\top}\cP\cD)^{-1}
(\cB^\top\eta+\cD^{\top}\zeta+\cD^{\top}\cP\sigma+\tilde{\rho})+\sigma\Big\}dW(t),\\
\ns\ds X^*(0) &=x,
\ea\right.$$
and  $(\eta,\zeta)$ is the adapted solution of the following backward stochastic differential equation,
$$\left\{\ba
\ds d\eta(t) &= -\Big\{\big[\cA^\top-(\cP\cB+\cC^\top \cP\cD+\cS^\top)(\cR
+\cD^\top \cP\cD)^{-1}\cB^\top\big]\eta\\
\ns\ds &\hp{=\ } +\big[\cC^\top-(\cP\cB+\cC^\top \cP\cD+\cS^\top)(\cR
+\cD^\top \cP\cD)^{-1}\cD^\top\big]\zeta\\
\ns\ns\ds &\hp{=\ } +\big[\cC^\top-(\cP\cB+\cC^\top \cP\cD+\cS^\top)
(\cR+\cD^\top \cP\cD)^{-1}\cD^\top\big]\cP\sigma\\
\ns\ds &\hp{=\ } -(\cP\cB+\cC^\top \cP\cD+\cS^\top)(\cR+\cD^\top \cP\cD)^{-1}\tilde{\rho}
+\cP b+\tilde{q}\Big\}dt \\
\ns\ds &\hp{=\ } + \zeta dW(t),\q t\in[0,T],\\
\eta(T) &=\tilde{g}.
\ea\right.$$
Furthermore, the value function at $x$ is given by
$$\cV(x)=\lan \cP(0)x,x\ran ,\q~\forall x\in\dbR^n.$$
\end{proposition}
\begin{proposition}
Assume that
\begin{equation}\label{2.2}
\cG\gg0,\q\cR\gg0,\q\cQ-\cS^\top\cR^{-1}\cS\gg0.
\end{equation}
Then the uniformly convex condition \rf{2.1} holds for a constant $\a>0$, and the solution of Riccati equation \rf{Riccati-SDE} satisfies
$$\cP(t)\ges0,\q~\forall t\in[0,T].$$
Moreover, if in addition to \rf{2.2}, $\cG>0$, then the solution $\cP(t)>0$ for all $t\in[0,T]$.
\end{proposition}


\section{A characterization of optimal controls in terms of FBSDEs}

We now present a characterization of the optimal control in terms of forward-backward stochastic differential equations, which will be used to prove the control constructed later to be optimal.
\begin{theorem}\label{Thm3.1}
Let (A1) and (A2) hold and let the terminal state  $\xi\in L^2_{\cF_T}(\Om;\dbR^n)$ be given.
A control $u^*\in\cU$ is optimal for $\xi$ if and only if the following two conditions hold:
\begin{itemize}
  \item [(i)] $J^0(0;u)\ges0$ for all $u\in\cU$.
  \item [(ii)] The adapted solution $(X^*,Y^*,Z^*)$ to the following decoupled FBSDE
$$\left\{\ba
\ds dX^*(t) &= \big[-A^\top(t)X^*(t)+Q(t)Y^*(t)+S^\top_1(t)Z^*(t)+S^\top_2(t)u^*(t)+q(t)\big]dt\\
\ns\ds &\hp{=\ } +\big[-C(t)X^*(t)+S_1(t)Y^*(t)
+R_{11}(t)Z^*(t)+R_{12}(t)u^*(t)+\rho_1(t)\big]dW(t),\\
\ns\ds dY^*(t) &= \big[A(t)Y^*(t)+B(t)u^*(t)+C(t)Z^*(t)+f(t)\big]dt+Z^*(t)dW(t),\ \ t\in[0,T],\\
\ns\ds X^*(0) &= GY^*(0)+g,\q Y^*(T)=\xi,
\ea\right.$$
\end{itemize}
satisfies
\begin{align}\label{3.2}
  S_2(t)Y^*(t)+R_{21}(t)Z^*(t)-B^\top(t)X^*(t)+R_{22}u^*(t) +\rho_2(t)=0,\q~ t\in[0,T].
\end{align}
\end{theorem}

\begin{proof}
Note that $u^*\in\cU$ is optimal for $\xi$ if and only if
\begin{align}\label{3.5}
  J(\xi;u^*+\e u)-J(\xi;u^*)\ges0,\q~\forall u\in\cU,\ \forall\e\in\dbR.
\end{align}
Let $u\in\cU$ and $\e\in\dbR$ be fixed but arbitrary. Denote by $(Y,Z)$ the adapted solution of the following BSDE,
$$\left\{\ba
\ds dY(t) &=\big[A(t)Y(t)+B(t)u(t)+C(t)Z(t)\big]dt+Z(t)dW(t), \q t\in[0,T],\\
\ns\ds Y(T) &=0,
\ea\right.$$
and denote by $(Y^\e,Z^\e)$ the adapted solution of
$$\left\{\ba
\ds dY^\e(t) &=\big[A(t)Y^\e(t)+B(t)[u^*(t)+\e u(t)]+C(t)Z^\e(t)+f(t)\big]dt+Z^\e(t)dW(t),\q t\in[0,T],\\
\ns\ds Y(T) &=\xi.
\ea\right.$$
By the uniqueness of the adapted solution of BSDEs, it is clearly that
$(Y^\e,Z^\e)=(Y^*+\e Y,Y^*+\e Z)$.
Therefore,
\begin{align}\label{3.3}
\ds & J(\xi;u^*+\e u)-J(\xi,u^*) \nonumber\\
\ns\ds &\q=2\e\dbE\Bigg\{\big\lan GY^*(0),Y(0)\big\ran+\big\lan g,Y(0)\big\ran \nonumber\\
\ns\ds &\q\hp{=\ } +\int_0^T\Bigg[
\llan\begin{pmatrix}Q(t)&S_1^\top(t)&S_2^\top(t)\\S_1(t)&R_{11}(t)&R_{12}(t)\\S_2(t)&R_{21}(t)&R_{22}(t) \end{pmatrix}
\begin{pmatrix}Y^*(t)\\Z^*(t)\\u^*(t)\end{pmatrix},
\begin{pmatrix}Y(t)\\Z(t)\\u(t)\end{pmatrix}\rran
 +\llan\begin{pmatrix}q(t)\\ \rho_1(t)\\ \rho_2(t) \end{pmatrix},
   \begin{pmatrix}Y(t)\\Z(t)\\u(t)\end{pmatrix}\rran\Bigg] ds\Bigg\} \nonumber\\
\ns\ds &\q\hp{=\ }+\e^2\dbE\Bigg\{\big\lan GY(0),Y(0)\big\ran
+\int_0^T\Bigg[
\llan\begin{pmatrix}Q(t)&S_1^\top(t)&S_2^\top(t)\\S_1(t)&R_{11}(t)&R_{12}(t)\\S_2(t)&R_{21}(t)&R_{22}(t) \end{pmatrix}
\begin{pmatrix}Y(t)\\Z(t)\\u(t)\end{pmatrix},
\begin{pmatrix}Y(t)\\Z(t)\\u(t)\end{pmatrix}\rran  \nonumber\\
\ns\ds &\q=2\e\dbE\Bigg\{\big\lan GY^*(0)+g,Y(0)\big\ran \nonumber\\
\ns\ds &\q\hp{=\ } +\int_0^T\Bigg[
\llan\begin{pmatrix}Q(t)&S_1^\top(t)&S_2^\top(t)\\S_1(t)&R_{11}(t)&R_{12}(t)\\S_2(t)&R_{21}(t)&R_{22}(t) \end{pmatrix}
\begin{pmatrix}Y^*(t)\\Z^*(t)\\u^*(t)\end{pmatrix},
\begin{pmatrix}Y(t)\\Z(t)\\u(t)\end{pmatrix}\rran
 +\llan\begin{pmatrix}q(t)\\ \rho_1(t)\\ \rho_2(t) \end{pmatrix},
   \begin{pmatrix}Y(t)\\Z(t)\\u(t)\end{pmatrix}\rran\Bigg] ds\Bigg\} \nonumber\\
\ns\ds &\q\hp{=\ } + \e^2 J^0(0;u).
\end{align}
Using the integration by parts formula to $\big\lan X^*,Y\big\ran$, we have
\begin{equation}\label{3.4}
\begin{aligned}
\ds -\big\lan GY^*(0)+g,Y(0)\big\ran
  &=\dbE\int_0^T\big[\lan QY^*+S^\top_1Z^*+S^\top_2u^*+q,Y\ran \\
\ns\ds &\hp{=\ } +\lan S_1Y^*+R_{11}Z^*+R_{12}u^*+\rho_1,Z\ran +\lan B^\top X^*,u\ran  \big]dt.
\end{aligned}
\end{equation}
Substitute \rf{3.4} in \rf{3.3} yields that
\begin{align*}
  J(\xi;u^*+\e u)-J(\xi,u^*)=\e^2J^0(0,u)
  +2\e\dbE\int_0^T \lan S_2Y^*+R_{21}Z^*-B^\top X^*+R_{22}u^*+\rho_2,u\ran dt.
\end{align*}
From the above, it is easy to see that \rf{3.5} holds if and only if \rf{3.2} holds and $J^0(0,u)\ges0$ for every $u\in\cU$.
\end{proof}

\section{Construction of optimal controls}

In this section we construct the optimal control of Problem (BSLQ) under the following uniform
positivity condition:
\begin{itemize}
  \item [\bf{(A3)}] There is a constant $\d>0$ such that
      \begin{equation}\label{A3}
        J^0(0;u)\ges\d\dbE\int_0^T|u(t)|^2dt,\q~\forall u\in\cU.
      \end{equation}
\end{itemize}
First, we observe that the uniform positivity condition (A3) implies $R_{22}\gg0$ (see Remark 5.4 of Sun--Wu--Xiong \cite{Sun-Wu-Xiong2021}).
Then, for simplicity presentation, we denote
\begin{equation}\label{4.21}
\begin{aligned}
\ds \mathscr{S}_{1} &= S_{1}-R_{12}R_{22}^{-1}S_{2},
&\mathscr{R}_{11}&=R_{11}-R_{12}R_{22}^{-1}R_{21}, \\
\ns\ds \mathscr{C}     &= C-B R_{22}^{-1} R_{21}, & v &=u+R_{22}^{-1} R_{21} Z.
\end{aligned}
\end{equation}
Using the notations \rf{4.21} and noting that $R_{22}\gg0$, it is easy to check that the original Problem (BSLQ) is equivalent to the following backward stochastic LQ problem with the state equation
\begin{equation}\label{4.22}\left\{\ba
\ds dY(t) &=[A(t) Y(t)+B(t) v(t)+\mathscr{C}(t) Z(t) +f(t)] d t+Z(t) d W(t),\q~t\in[0,T], \\
\ns\ds  Y(T) &=\xi,
\ea\right.\end{equation}
and the cost functional
\begin{align}\label{4.23}
\ds J\big(\xi;u\big)
&\deq\dbE\Bigg\{\big\lan GY(0),Y(0)\big\ran+2\big\lan g,Y(0)\big\ran \nonumber\\
\ns\ds &\hp{=\ } +\int_0^T\Bigg[
\llan\begin{pmatrix}Q(t)&\mathscr{S}_1^\top(t)&S_2^\top(t)\\\mathscr{S}_1(t)&\mathscr{R}_{11}(t)&0
\\S_2(t)&0&R_{22}(t) \end{pmatrix}
\begin{pmatrix}Y(t)\\Z(t)\\v(t)\end{pmatrix},
\begin{pmatrix}Y(t)\\Z(t)\\v(t)\end{pmatrix}\rran \\
\ns\ds &\hp{=\ }  +2\llan\begin{pmatrix}q(t)\\ \rho_1(t)\\ \rho_2(t) \end{pmatrix},
   \begin{pmatrix}Y(t)\\Z(t)\\u(t)\end{pmatrix}\rran\Bigg] dt\Bigg\}. \nonumber
\end{align}
Moreover, we let $H\in C([0,T];\dbS^n)$ be the unique solution of the following linear ordinary differential equation
\begin{equation*}\left\{\ba
\ds & \dot{H}(t)+H(t) A(t)+A(t)^{\top} H(t)+Q(t)=0, \q  t \in[0,T], \\
\ns\ds & H(0)=-G.
\ea\right.
\end{equation*}
Applying the integration by parts formula to $\lan HY,Y\ran$ on $[0,T]$, where $Y$ is the state process determined by \rf{4.22}, we have that
\begin{align*}
\ds &\dbE \lan H(T)\xi, \xi\ran  + \dbE\lan G Y(0), Y(0)\ran \\
\ns\ds &\q= \dbE  \int_{0}^{T}\left[\big\lan (\dot{H}+H A+A^{\top} H) Y, Y\big\ran
+2\lan B^{\top} H Y, v\ran+2\lan\mathscr{E}^{\top} H Y, Z\ran +2\lan HY,f\ran
+\lan H Z, Z\ran\right] d t \\
\ns\ds &\q=\dbE  \int_{0}^{T}\left[-\lan QY, Y\ran+2 \lan B^{\top} H Y, v \ran+2 \lan\mathscr{C}^{\top} H Y, Z \ran
+2\lan HY,f\ran+\lan H Z, Z\ran\right] d t \\
\ns\ds &\q=\dbE \int_{0}^{T}\left\{\left\lan\left(\begin{array}{ccc}-Q & H \mathscr{C} & H B \\ \mathscr{C}^{\top} H & H & 0 \\
B^{\top} H & 0 & 0
\end{array}\right)\left(\begin{array}{c}
Y \\
Z \\
v
\end{array}\right),\left(\begin{array}{l}
Y \\
Z \\
v
\end{array}\right)\right\ran +2\lan HY,f\ran \right\} dt .
\end{align*}
Substituting for the term $\dbE \lan G Y(0), Y(0)\ran$ in the cost functional \rf{4.23} yields that
\begin{align}
\ds J\big(\xi;u\big)=&~\dbE\Bigg\{-\lan H(T) \xi, \xi\ran+2\big\lan g,Y(0)\big\ran \nonumber\\
\ns\ds & +\int_0^T\Bigg[
\llan\begin{pmatrix}0&(\cS^H_1)^\top&(\cS^H_2)^\top\\ \cS^H_1&\cR^H_{11}&0
\\ \cS^H_2&0&R_{22} \end{pmatrix}
\begin{pmatrix}Y\\Z\\v\end{pmatrix},
\begin{pmatrix}Y\\Z\\v\end{pmatrix}\rran
+2\llan\begin{pmatrix}\ q^H\\ \rho_1\\ \rho_2 \end{pmatrix},
   \begin{pmatrix}Y\\Z\\u\end{pmatrix}\rran\Bigg] ds\Bigg\}, \nonumber
\end{align}
where
$$
\cS_{1}^{H}=\mathscr{S}_{1}+\mathscr{C}^{\top} H, \quad \cS_{2}^{H}=S_{2}+B^{\top} H, \quad \cR_{11}^{H}=\mathscr{R}_{11}+H,\quad  q^H= q +Hf.
$$
So, for a given terminal state $\xi$, minimizing $J(\xi;u)$ subject to \rf{state} is equivalent to minimizing the following cost functional
\begin{align*}
\ds J\big(\xi;u\big)=&~\dbE\Bigg\{2\big\lan g,Y(0)\big\ran\\
\ns\ds & +\int_0^T\Bigg[
\llan\begin{pmatrix}0&(\cS^H_1)^\top&(\cS^H_2)^\top\\ \cS^H_1&\cR^H_{11}&0
\\ \cS^H_2&0&R_{22} \end{pmatrix}
\begin{pmatrix}Y\\Z\\v\end{pmatrix},
\begin{pmatrix}Y\\Z\\v\end{pmatrix}\rran
+2\llan\begin{pmatrix}\tilde{q}\\ \rho_1\\ \rho_2 \end{pmatrix},
   \begin{pmatrix}Y\\Z\\u\end{pmatrix}\rran\Bigg] ds\Bigg\}, \nonumber
\end{align*}
subject to the state equation \rf{4.22},  which enables us to simplify Problem (BSLQ) firstly by assuming
\begin{align}\label{4.1}
  G=0,\q~ Q(t)=0,\q~ R_{12}=R^\top_{21}=0,\q~\forall t\in[0,T].
\end{align}
Therefore, in the rest of this section we would like to first discuss the case of \rf{4.1} holds, and then present the general result.
\subsection{The case of (\ref{4.1})}\label{subsec4.1}

Under the condition \rf{4.1}, the initial Problem (BSLQ) is equivalent to minimizing the following cost functional
\begin{equation}\label{4.2}
\begin{aligned}
\ds J\big(\xi;u\big)=&~\dbE\Bigg\{2\big\lan g,Y(0)\big\ran\\
\ns\ds& +\int_0^T\Bigg[
\llan\begin{pmatrix}0&S_1^\top(t)&S_2^\top(t)\\S_1(t)&R_{11}(t)&0\\S_2(t)&0&R_{22}(t) \end{pmatrix}
\begin{pmatrix}Y(t)\\Z(t)\\u(t)\end{pmatrix},
\begin{pmatrix}Y(t)\\Z(t)\\u(t)\end{pmatrix}\rran \\
\ns\ds & +2\llan\begin{pmatrix}q(t)\\ \rho_1(t)\\ \rho_2(t) \end{pmatrix},
   \begin{pmatrix}Y(t)\\Z(t)\\u(t)\end{pmatrix}\rran\Bigg] ds\Bigg\},
\end{aligned}
\end{equation}
subject to the initial state equation \rf{state}. Next, in order to construct the optimal control of Problem (BSLQ), we introduce the following Riccati equation
\begin{equation}\label{Riccati}
\left\{\begin{aligned}
\ds & \dot{\Si}(t) -A(t) \Si(t)-\Si(t) A(t)^{\top}+\mathcal{B}(t, \Si(t))\left[R_{22}(t)\right]^{-1} \mathcal{B}(t,\Si(t))^{\top} \\
\ns\ds &\hp{\dot{\Si}(t)} +\mathcal{C}(t, \Si(t))[\mathcal{R}(t, \Si(t))]^{-1} \Si(t) \mathcal{C}(t, \Si(t))^{\top}=0, \q t\in[0,T], \\
\ns\ds & \Si(T)=0 .
\end{aligned}\right.
\end{equation}
where $\Si:[0, T] \rightarrow \mathbb{S}^{n}$ is an $\mathbb{S}^{n}$-valued function, and
\begin{align*}
\ds \mathcal{B}(t, \Si(t)) &= B(t)+\Si(t) S_{2}(t)^{\top}, \\
\ns\ds \mathcal{C}(t, \Si(t)) &= C(t)+\Si(t) S_{1}(t)^{\top}, \\
\ns\ds \mathcal{R}(t, \Si(t)) &= I+\Si(t) R_{11}(t).
\end{align*}
When there is no risk for confusion, in the following for simplicity presentation, we would like to frequently suppress the argument $t$ from our notations and write $\mathcal{B}(t,\Si(t))$, $\mathcal{C}(t,\Si(t))$ and $\mathcal{R}(t,\Si(t))$ as $\mathcal{B}(\Si)$, $\mathcal{C}(\Si)$, and $\mathcal{R}(\Si)$, respectively. For Riccati equation \rf{Riccati}, we have the following result concerning the existence and uniqueness, which essentially is Theorem 6.2 of Sun--Wu--Xiong \cite{Sun-Wu-Xiong2021}.

\begin{proposition}
Let (A1)-(A3) and \rf{4.1} hold. Then the Riccati equation \rf{Riccati} admits a unique positive semidefinite solution $\Si \in C\left([0, T] ; \mathbb{S}^{n}\right)$ such that $\mathcal{R}(\Si)$ is invertible a.e. on $[0, T]$ and $\mathcal{R}(\Si)^{-1} \in L^{\infty}\left(0, T ; \mathbb{R}^{n}\right)$.
\end{proposition}
With the solution $\Si$ to the Riccati equation \rf{Riccati}, before constructing the optimal control of Problem (BSLQ), we further introduce the following linear BSDE:
\begin{equation}\label{4.3}\left\{\begin{aligned}
\ds d\varphi(t) &= \a(t,\Si(t)) d t+\beta(t) d W(t), \quad t \in[0, T], \\
\ns\ds  \varphi(T) &= \xi,
\end{aligned}\right.
\end{equation}
where
\begin{equation}\label{4.9}
\begin{aligned}
\ds \a(\Si)=&~ \Big\{ \left[A-\mathcal{B}(\Si) R_{22}^{-1} S_{2}-\mathcal{C}(\Si) \mathcal{R}(\Si)^{-1} \Si S_{1}\right] \varphi+\mathcal{C}(\Si) \mathcal{R}(\Si)^{-1} \beta \\
\ns\ds &~ -C \mathcal{R}(\Si)^{-1}\Si\rho_1
-\Si S_{1}^{\top} \mathcal{R}(\Si)^{-1} \Si\rho_1
-BR_{22}^{-1}\rho_2
-\Si S_{2}^{\top} R_{22}^{-1}\rho_2+\Si q+f \Big\} \\
\ns\ds =&~ \Big\{ \left[A-\mathcal{B}(\Si) R_{22}^{-1} S_{2}-\mathcal{C}(\Si) \mathcal{R}(\Si)^{-1} \Si S_{1}\right] \varphi+\mathcal{C}(\Si) \mathcal{R}(\Si)^{-1} \beta \\
\ns\ds &~ -\cC(\Si) \mathcal{R}(\Si)^{-1}\Si\rho_1-\cB(\Si) R_{22}^{-1}\rho_2+\Si q+f \Big\}.
\end{aligned}
\end{equation}
In terms of the solution $\Si$ to the Riccati equation \rf{Riccati} and the adapted solution
$(\f,\b)$ to the BSDE \rf{4.3}, we now can construct the optimal control of Problem (BSLQ) as
follows.
\begin{theorem}\label{Thm4.2}
Let (A1)-(A3) and \rf{4.1} hold. Let $(\f,\b)$ be the adapted solution to the
BSDE \rf{4.3} and $X$ the solution to the following SDE:
\begin{equation}\label{4.4}\left\{\ba
\ds dX(t) &=\Big\{\big[S_{1}^{\top} \mathcal{R}(\Si)^{-1} \Si\mathcal{C}(\Si)^{\top}+S_{2}^{\top} R_{22}^{-1} \mathcal{B}(\Si)^{\top}-A^{\top}\big] X\\
\ns\ds &\hp{=\ } -\big[S_{1}^{\top} \mathcal{R}(\Si)^{-1}\Si S_{1}+S_{2}^{\top}R_{22}^{-1}S_{2}\big]\varphi
+S_{1}^{\top} \cR(\Si)^{-1} \beta-S_{1}^{\top} \cR(\Si)^{-1} \Si\rho_1-S_{2}^{\top} R_{22}^{-1}\rho_2+q \Big\} d t \\
\ns\ds &\hp{=\ } - \big[\cR(\Si)^{-1}\big]^{\top} \big[\mathcal{C}(\Si)^{\top} X-S_{1} \varphi-R_{11} \beta -\rho_1\big]  d W(t), \\
\ns\ds X(0)&= g .
\end{aligned}\right.
\end{equation}
Then the optimal control of Problem (BSLQ) for the terminal state $\xi$ is given by
\begin{equation}\label{4.5}
u(t)=\left[R_{22}(t)\right]^{-1}\big[\mathcal{B}(t, \Si(t))^{\top} X(t)-S_{2}(t) \varphi(t)
-\rho_2(t) \big], \quad t \in[0, T] .
\end{equation}
\end{theorem}

\begin{proof}
Let us define for $t\in[0,T]$,
\begin{align}
\ds    Y(t) &= -\Si(t)X(t)+\f(t), \label{4.6}\\
\ns\ds Z(t) &=\cR(t,\Si(t))^{-1}\big[\Si(t)\cC(t,\Si(t))^{\top} X(t)-\Si(t) S_{1}(t)\varphi(t) -\Si(t)\rho_1(t) +\beta(t)\big].  \label{4.7}
\end{align}
We observe that
\begin{align}
\ds R_{22}u &= \mathcal{B}(\Si)^{\top} X-S_{2} \varphi -\rho_2 \nonumber\\
\ns\ds &= (B+\Si S_{2}^{\top})^\top X-S_{2} \varphi -\rho_2 \nonumber\\
\ns\ds &= B^{\top} X +S_{2}(\Si X- \varphi) -\rho_2 \nonumber\\
\ns\ds &= B^{\top} X -S_{2}Y -\rho_2. \label{4.8}
\end{align}
Furthermore, using \rf{4.5} and \rf{4.7} we obtain
\begin{align*}
\ds &S_{1}^{\top} Z+S_{2}^{\top} u\\
\ns\ds &= S_{1}^{\top} \mathcal{R}(\Si)^{-1}\big[\Si\mathcal{C}(\Si)^{\top} X-\Si S_{1} \varphi -\Si\rho_1 +\beta\big]
+S_{2}^{\top} R_{22}^{-1} \mathcal{B}(\Si)^{\top} X
-S_{2}^{\top} R_{22}^{-1} S_{2} \varphi-S_{2}^{\top} R_{22}^{-1}\rho_2 \\
\ns\ds &=\big[S_{1}^{\top} \mathcal{R}(\Si)^{-1} \Si \mathcal{C}(\Si)^{\top}+S_{2}^{\top} R_{22}^{-1} \mathcal{B}(\Si)^{\top}\big] X
-\big[S_{1}^{\top} \mathcal{R}(\Si)^{-1} \Si S_{1}+S_{2}^{\top} R_{22}^{-1} S_{2}\big]\varphi\\
\ns\ds &\q +S_{1}^{\top} \mathcal{R}(\Si)^{-1} \beta-S_{1}^{\top} \mathcal{R}(\Si)^{-1} \Si\rho_1-S_{2}^{\top} R_{22}^{-1}\rho_2,
\end{align*}
from which it implies that
\begin{equation}\label{4.0}\begin{aligned}
-A^\top X+S_{1}^{\top} Z+S_{2}^{\top} u +q
&=\big[S_{1}^{\top} \mathcal{R}(\Si)^{-1} \Si \mathcal{C}(\Si)^{\top}+S_{2}^{\top} R_{22}^{-1} \mathcal{B}(\Si)^{\top}-A^\top\big] X \\
\ns\ds &\q -\big[S_{1}^{\top} \mathcal{R}(\Si)^{-1} \Si S_{1}+S_{2}^{\top} R_{22}^{-1} S_{2}\big] \varphi
+S_{1}^{\top} \mathcal{R}(\Si)^{-1} \beta\\
\ns\ds &\q -S_{1}^{\top} \mathcal{R}(\Si)^{-1} \Si\rho_1
-S_{2}^{\top} R_{22}^{-1}\rho_2+q.
\end{aligned}\end{equation}
Similarly, using \rf{4.6} and \rf{4.7} we obtain
\begin{equation*}\begin{aligned}
\ds -C^{\top} X+S_{1} Y+R_{11} Z +\rho_1
=&-\big(C^{\top}+S_{1} \Si\big) X+S_{1} \varphi+R_{11} Z+\rho_1 \\
%
%
\ns\ds =&\left[R_{11} \mathcal{R}(\Si)^{-1} \Si-I\right] \mathcal{C}(\Si)^{\top} X+\left[I-R_{11} \mathcal{R}(\Si)^{-1} \Si\right] S_{1} \varphi \\
\ns\ds &+R_{11} \mathcal{R}(\Si)^{-1} \beta  +(I-R_{11} \mathcal{R}(\Si)^{-1}\Si)\rho_1.
\end{aligned}\end{equation*}
Note that
\begin{equation*}
\begin{array}{l}
\ds R_{11} \mathcal{R}(\Si)^{-1}=R_{11}\left(I+\Si R_{11}\right)^{-1}
=\left(I+R_{11} \Si\right)^{-1} R_{11}, \\
\ns\ds I-R_{11} \mathcal{R}(\Si)^{-1} \Si=\left(I+R_{11} \Si\right)^{-1}=\left[\mathcal{R}(\Si)^{-1}\right]^{\top},
\end{array}
\end{equation*}
we further obtain that
\begin{equation}\label{4.10}
-C^{\top} X+S_{1} Y+R_{11} Z  +\rho_1 =
-\big[\mathcal{R}(\Si)^{-1}\big]^{\top}\big[\mathcal{C}(\Si)^{\top} X-S_{1} \varphi-R_{11}\beta -\rho_1\big].
\end{equation}
This implies that the solution of \rf{4.4} satisfies the following equation
\begin{equation}\label{4.11}\left\{\ba
\ds d X(t) &= \big(-A^{\top} X+S_{1}^{\top} Z+S_{2}^{\top} u +q\big) d t+\big(-C^{\top} X+S_{1} Y+R_{11} Z +\rho_1\big) d W, \\
\ns\ds X(0) &= g.
\ea\right.
\end{equation}
Applying It\^{o}'s formula to \rf{4.6}, we have
$$\begin{aligned}
\ds dY &=-\dot{\Si} X d t-\Si d X+d \varphi \\
\ns\ds &=\big[\alpha(\Si)-\dot{\Si} X-\Si \big(-A^{\top} X+S_{1}^{\top} Z+S_{2}^{\top} u +q \big)\big] d t\\
\ns\ds &\hp{=\ } +\big[\beta-\Si\big(-C^{\top} X+S_{1} Y+R_{11} Z +\rho_1 \big)\big] d W.
\end{aligned}$$
Using \rf{4.0} and \rf{4.9}, and note that $\Si$ satisfies Riccati equation \rf{Riccati}, we obtain
$$
\begin{aligned}
\ds & \alpha(\Si)- \dot{\Si} X-\Si\big(-A^{\top} X+S_{1}^{\top} Z+S_{2}^{\top} u +q \big) \\
\ns\ds &= \alpha(\Si)-\big[\dot{\Si}-\Si A^{\top}+\Si S_{1}^{\top} \mathcal{R}(\Si)^{-1} \Si \mathcal{C}(\Si)^{\top}+\Si S_{2}^{\top} R_{22}^{-1} \mathcal{B}(\Si)^{\top}\big] X \\
\ns\ds &\q +\Si \big[S_{1}^{\top} \mathcal{R}(\Si)^{-1} \Si S_{1}+S_{2}^{\top}
 R_{22}^{-1} S_{2}\big] \varphi-\Si S_{1}^{\top} \mathcal{R}(\Si)^{-1}\beta\\
\ns\ds &\q +\Si S_{1}^{\top} \mathcal{R}(\Si)^{-1} \Si\rho_1
+\Si S_{2}^{\top} R_{22}^{-1}\rho_2-\Si q \\
\ns\ds &= \alpha(\Si)-\big[A \Si-C \mathcal{R}(\Si)^{-1} \Si \mathcal{C}(\Si)^{\top}-B R_{22}^{-1} \mathcal{B}(\Si)^{\top}\big] X \\
\ns\ds &\q +\Si\big[S_{1}^{\top} \mathcal{R}(\Si)^{-1} \Si S_{1}+S_{2}^{\top} R_{22}^{-1} S_{2}\big] \varphi-\Si S_{1}^{\top} \mathcal{R}(\Si)^{-1} \beta\\
\ns\ds &\q +\Si S_{1}^{\top} \mathcal{R}(\Si)^{-1} \Si\rho_1
+\Si S_{2}^{\top} R_{22}^{-1}\rho_2-\Si q \\
\ns\ds &= A Y+\big[C \mathcal{R}(\Si)^{-1} \Si \mathcal{C}(\Si)^{\top}+B R_{22}^{-1} \mathcal{B}(\Si)^{\top}\big] X-B R_{22}^{-1} S_{2} \varphi \\
\ns\ds &\q -C \mathcal{R}(\Si)^{-1}(\Si S_{1} \varphi-\beta)
  -C \mathcal{R}(\Si)^{-1}\Si(t)\rho_1(t)
  -BR_{22}^{-1}\rho_2 +f \\
\ns\ds & = A Y+B R_{22}^{-1}\big[\mathcal{B}(\Si)^{\top} X-S_{2} \varphi -\rho_2\big]+
C \mathcal{R}(\Si)^{-1}\big[\Si \mathcal{C}(\Si)^{\top} X-\Si S_{1} \varphi
-\Si(t)\rho_1(t)
+\beta\big]
+f \\
\ns\ds &= A Y+B u+C Z +f.
\end{aligned}
$$
Similarly, using \rf{4.10} and the following relation
\begin{align*}
\ds \Si\left[\mathcal{R}(\Si)^{-1}\right]^{\top} &=\Si\left(I+R_{11} \Si\right)^{-1}=\left(I+\Si R_{11}\right)^{-1} \Si=\mathcal{R}(\Si)^{-1} \Si, \\
\ns\ds I-\mathcal{R}(\Si)^{-1} \Si R_{11} &=I-\left(I+\Si R_{11}\right)^{-1} \Si R_{11}=\left(I+\Si R_{11}\right)^{-1}=\mathcal{R}(\Si)^{-1},
\end{align*}
we have that
\begin{align*}
\ds & \beta-\Si \big(-C^{\top} X+S_{1} Y+R_{11} Z +\rho_1\big) \\
\ns\ds &=\beta+\Si\left[\mathcal{R}(\Si)^{-1}\right]^{\top}\big[\mathcal{C}(\Si)^{\top} X-S_{1} \varphi-R_{11} \beta  -\rho_1 \big] \\
\ns\ds &=\mathcal{R}(\Si)^{-1}\big[\Si \mathcal{C}(\Si)^{\top} X-\Si S_{1} \varphi
 -\Si\rho_1 \big]+\big[I-\mathcal{R}(\Si)^{-1} \Si R_{11}\big] \beta \\
\ns\ds &=\mathcal{R}(\Si)^{-1}\big[\Si \mathcal{C}(\Si)^{\top} X-\Si S_{1} \varphi
 -\Si\rho_1+ \beta \big]\\
\ns\ds &=Z.
\end{align*}
Therefore, the pair $(Y,Z)$ defined by \rf{4.6} and \rf{4.7} satisfies the following BSDE:
\begin{equation}\label{4.12}\left\{\ba
\ds dY(t) &=(A Y+B u+C Z+ f ) d t+Z d W, \\
\ns\ds  Y(T) &=\xi.
\ea\right.
\end{equation}
Combining \rf{4.11} and \rf{4.12}, we see that the solution $X$ of \rf{4.4}, the pair $(Y,Z)$
defined by \rf{4.6} and \rf{4.7} satisfy the following FBSDE
\begin{equation}\label{FBSDE}\left\{\ba
\ds dX(t) &= (-A^{\top} X+S_{1}^{\top} Z+S_{2}^{\top} u +q  ) d t
+ (-C^{\top} X+S_{1} Y+R_{11} Z +\rho_1  ) d W, \\
\ns\ds dY(t) &=(A Y+B u+C Z +f ) d t+Z d W, \\
\ns\ds X(0) &=g, \q Y(T)=\xi.
\ea\right.
\end{equation}
In addition, combining \rf{4.8}, we have that the control $u$ defined by \rf{4.5} satisfies the following condition
\begin{equation}\label{4.13}
S_{2} Y-B^{\top} X+R_{22} u +\rho_2=0.
\end{equation}
Therefore, from Theorem \ref{Thm3.1}, we obtain that $u$ is the (unique) optimal control for the terminal state $\xi$.
\end{proof}

We conclude this section with a representation of the value function $V(\xi)$.

\begin{theorem}
Let (A1)-(A3) and \rf{4.1} hold. Then the value function of Problem (BSLQ) is given by
\begin{equation}\begin{aligned}
\ns\ds V(\xi)=&\ \dbE \bigg\{2\lan  \f(0),g\ran - \lan  \Si(0)g, g\ran\\
\ns\ds & +\int_{0}^{T} \Big[
 -\left\lan \mathcal{R}(\Si)^{-1}\Si\rho_1, \rho_1 \right\ran
 -\left\lan R_{22}^{-1}\rho_2,\rho_2 \right\ran
 +2\left\lan \mathcal{R}(\Si)^{-1}\beta, \rho_1 \right\ran \\
\ns\ds &+\left\lan  R_{11} \mathcal{R}(\Si)^{-1} \beta, \beta\right\ran
+2 \big\lan  S_{1}^{\top} \mathcal{R}(\Si)^{-1} \beta
-S_{1}^{\top} \mathcal{R}(\Si)^{-1} \Si\rho_1-S_{2}^{\top} R_{22}^{-1}\rho_2+q, \varphi \big\ran \\
\ns\ds &-\big\lan  [S_{1}^{\top} \mathcal{R}(\Si)^{-1} \Si S_{1}+S_{2}^{\top} R_{22}^{-1} S_{2} ] \varphi, \varphi\big\ran \Big] d t\bigg\}, \label{4.14}
\end{aligned}\end{equation}
where $(\f,\b)$ is the adapted solution of the BSDE \rf{4.3}.
\end{theorem}

\begin{proof}
Let $u$ be the optimal control for the terminal state $\xi$. Then, by Theorem \ref{Thm3.1}, the
adapted solution $(X,Y,Z)$ of \rf{FBSDE} satisfies \rf{4.13}. By the definition, we observe that
\begin{align*}
\ds V(\xi) &= J\big(\xi;u\big)=\dbE\Bigg\{2\big\lan g,Y(0)\big\ran\\
\ns\ds &\q +\int_0^T\Bigg[
\llan\begin{pmatrix}0&S_1^\top(t)&S_2^\top(t)\\S_1(t)&R_{11}(t)&0\\S_2(t)&0&R_{22}(t) \end{pmatrix}
\begin{pmatrix}Y(t)\\Z(t)\\u(t)\end{pmatrix},
\begin{pmatrix}Y(t)\\Z(t)\\u(t)\end{pmatrix}\rran
+2\llan\begin{pmatrix}q(t)\\ \rho_1(t)\\ \rho_2(t) \end{pmatrix},
   \begin{pmatrix}Y(t)\\Z(t)\\u(t)\end{pmatrix}\rran\Bigg] dt\Bigg\}\\
\ns\ds &= \dbE  \bigg\{2\big\lan g,Y(0)\big\ran+\int_{0}^{T}\big[2\left\lan  S_{1} Y, Z\right\ran +2\left\lan  S_{2} Y, u\right\ran +\left\lan  R_{11} Z, Z\right\ran +\left\lan  R_{22} u, u\right\ran \\
\ns\ds &\q
+2\left\lan  q,Y\right\ran  + 2\left\lan  \rho_1,Z\right\ran  + 2\left\lan  \rho_2,u\right\ran  \big] d t \bigg\}\\
\ns\ds &= \dbE  \bigg\{2\big\lan g,Y(0)\big\ran+
\int_{0}^{T}\big[\lan  S_{1}^{\top} Z+S_{2}^{\top} u, Y\ran +\lan  S_{1} Y+R_{11} Z, Z\ran +\lan  S_{2} Y+R_{22} u, u\ran \\
\ns\ds &\q
+2\left\lan  q,Y\right\ran + 2\lan\rho_1,Z\ran  + 2\lan\rho_2,u\ran\big] dt \bigg\}\\
\ns\ds &= \dbE  \bigg\{\big\lan 2g,Y(0)\big\ran+
\int_{0}^{T}\big[\lan  S_{1}^{\top} Z+S_{2}^{\top} u +2q, Y\ran +\lan  S_{1} Y+R_{11} Z +2\rho_1, Z\ran +\lan   B^{\top} X +\rho_2, u\ran
\big] d t \bigg\}.
\end{align*}
Take the integration by parts formula to $\left\lan X,Y\right\ran$ implies that
\begin{align*}
\ds &\dbE \lan  X(T), Y(T)\ran \\
\ns\ds &= \dbE \lan  g, Y(0)\ran
  +\dbE  \int_{0}^{T}\big[\lan  X, A Y+B u+C Z +f \ran
  +\lan -A^{\top} X+S_{1}^{\top} Z+S_{2}^{\top} u +q, Y\ran \\
\ns\ds & \q +\lan -C^{\top} X+S_{1} Y+R_{11} Z +\rho_1, Z\ran \big] d t \\
\ns\ds &=  \dbE \lan  g, Y(0)\ran
+\dbE  \int_{0}^{T}\big[\lan  X, Bu +f \ran
+\lan  S_{1}^{\top} Z+S_{2}^{\top} u +q , Y\ran
+\lan  S_{1} Y+R_{11} Z +\rho_1 , Z\ran \big] d t \\
\ns\ds &= V(\xi)- \dbE \lan g, Y(0)\ran
-\dbE  \int_{0}^{T}\big[
 -\lan  X,  f \ran
 +\lan  q , Y\ran
 +\lan  \rho_1 , Z\ran
 +\lan  \rho_2, u \ran
\big] dt.
\end{align*}
In other words, we have
$$\begin{aligned}
V(\xi)
=\dbE \lan  X(T), Y(T)\ran  + \dbE \lan  g, Y(0)\ran
+\dbE  \int_{0}^{T}\big[
 -\lan  X,  f \ran
 +\lan  q , Y\ran
 +\lan  \rho_1 , Z\ran
 +\left\lan  \rho_2, u\right\ran
\big] dt.
\end{aligned}$$
From Theorem \ref{Thm4.2}, we see that $X$ also satisfies the equation \rf{4.4} and $\f$ satisfies the equation \rf{4.3}. Applying the integration by parts formula to  $\left\lan X,\f\right\ran$, we obtain that
%
%
\begin{align*}
\ds & \dbE \lan X(T), \varphi(T)\ran -\dbE\lan g, \varphi(0)\ran \\
\ns\ds &= \dbE  \int_{0}^{T}\left\{\big\lan[S_{1}^{\top} \mathcal{R}(\Si)^{-1} \Si \mathcal{C}(\Si)^{\top}+S_{2}^{\top} R_{22}^{-1} \mathcal{B}(\Si)^{\top}-A^{\top}] X, \varphi\big\ran \right.\\
\ns\ds &\q -\big\lan [S_{1}^{\top} \mathcal{R}(\Si)^{-1} \Si S_{1}+S_{2}^{\top} R_{22}^{-1} S_{2}] \varphi, \varphi\big\ran
+\lan  S_{1}^{\top} \mathcal{R}(\Si)^{-1} \beta
-S_{1}^{\top} \mathcal{R}(\Si)^{-1} \Si\rho_1-S_{2}^{\top} R_{22}^{-1}\rho_2+q,
\varphi\ran  \\
\ns\ds &\q +\big\lan  X, [A-\mathcal{B}(\Si) R_{22}^{-1} S_{2}-\mathcal{C}(\Si) \mathcal{R}(\Si)^{-1} \Si S_{1} ] \varphi\big\ran \\
\ns\ds &\q+\big\lan  X, \mathcal{C}(\Si) \mathcal{R}(\Si)^{-1}\beta
-\cC(\Si) \mathcal{R}(\Si)^{-1}\Si\rho_1-\cB(\Si)R_{22}^{-1}\rho_2+\Si q+f\big\ran  \\
\ns\ds &\q \left.-\big\lan [\mathcal{R}(\Si)^{-1}]^{\top}[\mathcal{C}(\Si)^{\top} X-S_{1} \varphi-R_{11} \beta -\rho_1], \beta\big\ran \right\} d t \\
\ns\ds &= \dbE  \int_{0}^{T}\left\{\big\lan  R_{11} \mathcal{R}(\Si)^{-1} \beta
+\mathcal{R}(\Si)^{-1}\rho_1, \beta\big\ran
+\big\lan  2S_{1}^{\top} \mathcal{R}(\Si)^{-1} \beta
-S_{1}^{\top} \mathcal{R}(\Si)^{-1} \Si\rho_1-S_{2}^{\top} R_{22}^{-1}\rho_2+q, \varphi\big\ran \right.\\
\ns\ds&\q +\big\lan X,-\cC(\Si) \mathcal{R}(\Si)^{-1}\Si\rho_1-\cB(\Si)R_{22}^{-1}\rho_2+\Si q+f\big\ran
\left.-\big\lan [S_{1}^{\top} \mathcal{R}(\Si)^{-1}
\Si S_{1}+S_{2}^{\top} R_{22}^{-1} S_{2}] \varphi, \varphi\big\ran \right\} d t.
\end{align*}
Note that
$$ \dbE\lan  X(T), Y(T)\ran =\dbE \lan  X(T), \xi\ran =\dbE \lan  X(T), \varphi(T)\ran , $$
then we have
\begin{align*}
\ds V(\xi) &=\dbE \lan  X(T), \f(T)\ran  + \dbE \lan  g, Y(0)\ran
+\dbE  \int_{0}^{T}\big[  -\lan  X,  f \ran
 +\lan  q , Y\ran   +\lan  \rho_1 , Z\ran
 + \lan  \rho_2, u \ran  \big] dt\\
\ns\ds &= \dbE \lan  g, \f(0)\ran +\dbE \lan  g, Y(0)\ran
+\dbE  \int_{0}^{T}\big[  -\lan  X,  f \ran
 +\lan  q , Y\ran   +\lan  \rho_1 , Z\ran
 + \lan  \rho_2, u \ran  \big] dt\\
\ns\ds &\q+ \dbE  \int_{0}^{T}\bigg\{\lan  R_{11} \mathcal{R}(\Si)^{-1} \beta
+\mathcal{R}(\Si)^{-1}\rho_1, \beta \ran \\
\ns\ds &\q + \lan  2S_{1}^{\top} \mathcal{R}(\Si)^{-1} \beta
-S_{1}^{\top} \mathcal{R}(\Si)^{-1} \Si\rho_1-S_{2}^{\top} R_{22}^{-1}\rho_2+q, \varphi \ran \\
\ns\ds &\q + \lan  X, -\cC(\Si) \mathcal{R}(\Si)^{-1}\Si\rho_1
-\cB(\Si)R_{22}^{-1}\rho_2+\Si q+f \ran \\
\ns\ds &\q -\big\lan  [S_{1}^{\top} \mathcal{R}(\Si)^{-1}
\Si S_{1}+S_{2}^{\top} R_{22}^{-1} S_{2} ] \varphi, \varphi\big\ran \bigg\} d t.
\end{align*}
From Theorem \ref{Thm4.2}, we see that $Y$, $Z$ and $u$ have the following representation
\begin{align*}
\ds Y(t)&= -\Si X+\f,\\
\ns\ds Z(t)&= \mathcal{R}(\Si)^{-1}\big[\Si \mathcal{C}(\Si)^{\top} X-\Si S_{1} \varphi -\Si\rho_1 +\beta\big], \\
\ns\ds u(t)&= R_{22}^{-1}\big[\mathcal{B}(\Si)^{\top} X-S_{2} \varphi -\rho_2 \big].
\end{align*}
Moreover, note that
\begin{align*}
\ds \lan  X,\Si q\ran
&=\lan \Si X-\f, q\ran+\lan\f,q\ran=-\lan Y,q\ran +\lan\f,q\ran,\\
\ns\ds \left\lan X, \cB(\Si)R_{22}^{-1}\rho_2 \right\ran
&=\left\lan R_{22}^{-1}\left[\mathcal{B}(\Si)^{\top} X-S_{2} \varphi-\rho_2 \right],\rho_2 \right\ran
 +\left\lan R_{22}^{-1}\left(S_{2} \varphi+\rho_2 \right),\rho_2 \right\ran\\
\ns\ds &=\left\lan u,\rho_2\right\ran
 +\left\lan R_{22}^{-1}\left[S_{2} \varphi+\rho_2 \right],\rho_2 \right\ran,\\
\ns\ds \left\lan X, \cC(\Si) \mathcal{R}(\Si)^{-1}\Si\rho_1 \right\ran
&=\left\lan \mathcal{R}(\Si)^{-1}\left[\Si \mathcal{C}(\Si)^{\top} X
-\Si S_{1}\varphi -\Si\rho_1 +\beta\right],\rho_1 \right\ran \\
\ns\ds &\hp{=\ } +\left\lan\mathcal{R}(\Si)^{-1}\left[\Si S_{1} \varphi
+\Si\rho_1 -\beta\right],\rho_1 \right\ran\\
\ns\ds &=\left\lan Z,\rho_1\right\ran+\left\lan \mathcal{R}(\Si)^{-1}\left[\Si S_{1} \varphi+\Si\rho_1 -\beta\right],\rho_1 \right\ran.
\end{align*}
Then we obtain that
\begin{align*}
\ds V(\xi) &=2\dbE \lan g, \f(0)\ran-\dbE \lan \Si(0)g, g\ran\\
\ns\ds &\hp{=\ }  +\dbE  \int_{0}^{T} \Big\{\lan\f,q\ran
 -\left\lan R_{22}^{-1}\left[S_{2} \varphi+\rho_2 \right],\rho_2 \right\ran
 -\left\lan \mathcal{R}(\Si)^{-1}\left[\Si S_{1} \varphi+\Si\rho_1 -\beta\right],
 \rho_1 \right\ran \Big\} dt\\
\ns\ds &\hp{=\ } + \dbE\int_{0}^{T}\Big\{\left\lan  R_{11} \mathcal{R}(\Si)^{-1} \beta
+\mathcal{R}(\Si)^{-1}\rho_1, \beta\right\ran \\
\ns\ds &\hp{=\ }  +\big\lan  2S_{1}^{\top} \mathcal{R}(\Si)^{-1} \beta
-S_{1}^{\top} \mathcal{R}(\Si)^{-1} \Si\rho_1-S_{2}^{\top} R_{22}^{-1}\rho_2+q, \varphi\big\ran \\
\ns\ds &\hp{=\ }
-\big\lan[S_{1}^{\top} \mathcal{R}(\Si)^{-1} \Si S_{1}+S_{2}^{\top} R_{22}^{-1} S_{2}] \varphi, \varphi\big\ran \Big\} d t\\
\ns\ds &=\dbE\bigg\{2\lan\f(0),g\ran- \lan\Si(0)g,g\ran+\int_0^T \Big\{
 -\left\lan \mathcal{R}(\Si)^{-1}\Si\rho_1, \rho_1 \right\ran
 -\left\lan R_{22}^{-1}\rho_2,\rho_2 \right\ran
 +2\left\lan \mathcal{R}(\Si)^{-1}\beta, \rho_1 \right\ran \\
\ns\ds &\hp{=\ } +\left\lan R_{11} \mathcal{R}(\Si)^{-1}\beta, \beta\right\ran
+2\big\lan S_{1}^{\top} \mathcal{R}(\Si)^{-1} \beta
-S_{1}^{\top} \mathcal{R}(\Si)^{-1} \Si\rho_1-S_{2}^{\top} R_{22}^{-1}\rho_2+q, \varphi\big\ran\\
\ns\ds &\hp{=\ } -\big\lan [S_{1}^{\top} \mathcal{R}(\Si)^{-1}\Si S_{1}+S_{2}^{\top} R_{22}^{-1} S_{2} ] \varphi, \varphi\big\ran\Big\} dt\bigg\}.
\end{align*}
This completes the proof.
\end{proof}

\subsection{The general case.}

In this subsection, we generalize the results obtained above to the situation that without the condition \rf{4.1}. We shall only present the result, as the proof can be easily given using the argument at the beginning of Subsection \ref{subsec4.1} and the results established there for the condition \rf{4.1}.

\ms

Recall the following notations, for $t\in[0,T]$,
\begin{align*}
\ds \mathscr{C}(t) &=C(t)-B(t)\left[R_{22}(t)\right]^{-1} R_{21}(t), \\
\ns\ds \mathscr{S}_{1}(t) &=S_{1}(t)-R_{12}(t)\left[R_{22}(t)\right]^{-1} S_{2}(t), \\
\ns\ds \mathscr{R}_{11}(t) &=R_{11}(t)-R_{12}(t)\left[R_{22}(t)\right]^{-1} R_{21}(t).
\end{align*}
Let $H\in C([0,T];\dbS^n)$ be the unique solution to the following linear ordinary differential equation,
\begin{equation*}\left\{\ba
\ds & \dot{H}(t)+H(t) A(t)+A(t)^{\top} H(t)+Q(t)=0, \q t\in[0, T], \\
\ns\ds & H(0)=-G.
\ea\right.\end{equation*}
Moreover, for $t\in[0,T]$, denote
\begin{align*}
\ds S_{1}^{H}(t) &=\mathscr{S}_{1}(t)+\mathscr{C}(t)^{\top} H(t),
& \mathcal{B}^{H}(t, \Si(t)) &=B(t)+\Si(t)\lt[S_{2}^{H}(t)\rt]^{\top}, \\
\ns\ds S_{2}^{H}(t) &=S_{2}(t)+B(t)^{\top} H(t), & \mathcal{C}^{H}(t, \Si(t)) &=\mathscr{C}(t)+\Si(t)\left[S_{1}^{H}(t)\right]^{\top}, \\
\ns\ds R_{11}^{H}(t) &=\mathscr{R}_{11}(t)+H(t), & \mathcal{R}^{H}(t, \Si(t)) &=I+\Si(t) R_{11}^{H}(t),\\
\ns\ds q^H(t) &=q(t) +H(t)f(t).
\end{align*}

\begin{theorem}
Under the conditions (A1)-(A3), the following results hold.
\begin{itemize}
\item [(i)] Let $(\f,\b)$ be the adapted solution to the following BSDE
\begin{equation}\label{4.24}\left\{\ba
\ns\ds d\varphi(t) &= \Big\{ \left[A-\mathcal{B}^H(\Si) R_{22}^{-1} S_{2}^H-\mathcal{C}^H(\Si) [\mathcal{R}^H(\Si)]^{-1} \Si S_{1}^H\right] \varphi
+\cC^H(\Si) [\mathcal{R}^H(\Si)]^{-1} \beta \\
\ns\ds &\hp{=\ } -\cC^H(\Si) [\mathcal{R}^H(\Si)]^{-1}\Si\rho_1
-\cB^H(\Si) R_{22}^{-1}\rho_2+\Si q^H+f \Big\} d t+\beta(t) d W(t), \quad t \in[0, T], \\
\ns\ds \varphi(T) &= \xi,
\ea\right.\end{equation}
and let $X$ the solution to the following SDE
\begin{equation*}\left\{\begin{aligned}
\ds dX(t) &=\Big\{\big[(S_{1}^H)^{\top} [\cR^H(\Si)]^{-1} \Si [\cC^H(\Si)]^{\top}+(S_{2}^H)^{\top} R_{22}^{-1} [\cB^H(\Si)]^{\top}-A^{\top}\big] X\\
\ns\ds &\hp{=\ } -\big[(S_{1}^H)^{\top} [\cR^H(\Si)]^{-1}\Si S^H_{1}+(S_{2}^H)^{\top}R_{22}^{-1}S_{2}^H\big]\varphi
+(S_{1}^H)^{\top} [\cR^H(\Si)]^{-1} \beta\\
\ns\ds &\hp{=\ } -(S_{1}^H)^{\top} [\cR^H(\Si)]^{-1} \Si\rho_1-(S_{2}^H)^{\top} R_{22}^{-1}\rho_2+q^H \Big\} d t \\
\ns\ds &\hp{=\ } - \lt[\cR^H(\Si)^{-1}\rt]^{\top}\big[[\cC^H(\Si)]^{\top} X-S_{1}^H \varphi-R_{11} \beta -\rho_1\big] dW(t),\q t\in[0,T], \\
\ns\ds X(0)&= g .
\end{aligned}\right.
\end{equation*}
Then the optimal control of Problem (BSLQ) for the terminal state $\xi$ is given by
\begin{align*}
\ds u &= R_{22}^{-1}\Big\{\big[[\mathcal{B}^{H}(\Si)]^{\top}-R_{21} [\mathcal{R}^{H}(\Si)]^{-1} \Si [\mathcal{C}^{H}(\Si)]^{\top}\big] X\\
\ns\ds &\hp{=\ } +\big[R_{21} [\mathcal{R}^{H}(\Si)]^{-1} \Si S_{1}^{H}-S_{2}^{H}\big] \varphi
+R_{21} [\mathcal{R}^{H}(\Si)]^{-1}\Si\rho_1
-R_{21} [\mathcal{R}^{H}(\Si)]^{-1} \beta  -\rho_2 \Big\},
\end{align*}
where $\Si$ is the unique positive semidefinite solution of the following Riccati equation,
$$\left\{\ba
\ds & \dot{\Si} -A \Si-\Si A^{\top}+\mathcal{B}^{H}(\Si)\lt[R_{22}\rt]^{-1}\lt[\cB^{H}(\Si)\rt]^{\top} \\
\ns\ds &\hp{\dot{\Si}} +\cC^{H}(\Si)\left[\cR^{H}(\Si)\right]^{-1} \Si\left[\cC^{H}(\Si)\right]^{\top}=0, \\
\ns\ds & \Si(T) =0.
\ea\right.$$
\item [(ii)] The value function of Problem (BSLQ) is given by
\begin{equation*}\begin{aligned}
\ds V(\xi) &= \dbE\bigg\{-\lan H(T)\xi,\xi\ran+2\lan\f(0),g\ran-\lan\Si(0)g,g\ran \\
\ns\ds &\hp{=\ } +\int_{0}^{T} \Big\{-\left\lan[\cR^H(\Si)]^{-1}\Si\rho_1,\rho_1\right\ran
 -\left\lan R_{22}^{-1}\rho_2,\rho_2 \right\ran\\
\ns\ds &\hp{=\ }
 +2\left\lan\cR(\Si)^{-1}\beta, \rho_1 \right\ran +\left\lan R_{11} [\cR^H(\Si)]^{-1} \beta,\beta\right\ran\\
\ns\ds &\hp{=\ } +2\big\lan (S^H_{1})^{\top} [\cR^H(\Si)]^{-1} \beta
-(S^H_{1})^{\top} [\cR^H(\Si)]^{-1} \Si\rho_1-(S_{2}^H)^{\top} R_{22}^{-1}\rho_2+q^H, \varphi\big\ran\\
\ns\ds &\hp{=\ } -\big\lan\big[(S^H_{1})^{\top} [\cR^H(\Si)]^{-1} \Si S_{1}^H
+(S_{2}^H)^{\top} R_{22}^{-1} S^H_{2}\big] \varphi, \varphi\big\ran\Big\} d t\bigg\},
\end{aligned}\end{equation*}
where $(\f,\b)$ is the adapted solution of BSDE \rf{4.24}.
\end{itemize}
\end{theorem}

\section{Conclusion}

In this paper, we have investigated an indefinite backward stochastic linear-quadratic optimal
control problem with deterministic nonhomogeneous coefficients and have developed a general procedure
for constructing optimal controls.
The necessary and sufficient conditions of Problem (BSLQ) are derived for the solvability of the problem, and a characterization of the optimal control in terms of forward-backward stochastic differential equations has presented. The optimal control and the value function of Problem (BLSQ) are given out clearly.
The results obtained in this paper provide insight into some related topics, especially into
the study of zero-sum backward stochastic differential games.
In our future publication, we hope to report some relevant results along this line.


\begin{thebibliography}{90}
\addtolength{\itemsep}{-1.5ex}

\bibitem{Bismut1973}
J. Bismut,
Conjugate convex functions in optimal stochastic control,
J. Math. Anal. Appl., 44 (1973), 384--404.

\bibitem{Bi-Sun-Xiong2020}
X. Bi, J. Sun, and J. Xiong,
Optimal control for controllable stochastic linear systems,
ESAIM Control Optim. Calc. Var., 26 (2020), No. 98.

\bibitem{Du-Huang-Wu2018}
K. Du, J. Huang, and Z. Wu,
Linear quadratic mean-field-game of backward stochastic differential systems,
Math. Control Relat. Fields, 8 (2018), 653--678.

\bibitem{Davis1977}
M. Davis,
Linear Estimation and Stochastic Control,  Chapman and Hall, London, 1977.

\bibitem{Huang-Wang-Xiong2009}
J. Huang, G. Wang, and J. Xiong,
A maximum principle for partial information backward stochastic control problems with applications,
SIAM J. Control Optim., 48 (2009), 2106--2117.

\bibitem{Huang-Wang-Wu2016}
J. Huang, S. Wang, and Z. Wu,
Backward mean-field linear-quadratic-Gaussian (LQG) games: full and partial information,
IEEE Trans. Automat. Control, 61 (2016), 3784--3796.


\bibitem{Lim-Zhou2001}
A. E. B. Lim and X. Y. Zhou,
Linear-quadratic control of backward stochastic differential equations,
SIAM J. Control Optim., 40 (2001), 450--474.


\bibitem{Li-Sun-Xiong2019}
X. Li, J. Sun, and J. Xiong,
Linear quadratic optimal control problems for mean-field backward stochastic differential equations,
Appl. Math. Optim., 80 (2019),  223--250.


\bibitem{Ma-Yong1999}
J. Ma and J. Yong,
Forward-Backward Stochastic Differential Equations and Their Applications,
Lecture Notes in Math., Springer-Verlag, New York, 1999.

\bibitem{Pardoux-Peng1990}
E. Pardoux, S. Peng,
Adapted solution of a backward stochastic differential equation,
Syst. Control Lett., 14 (1990), 14 55--61.


\bibitem{Pham2009}
H. Pham,
Continuous-Time Stochastic Control and Optimization with Financial Applications,
Springer-Verlag, Berlin, 2009.


\bibitem{Peng2010}
S. Peng,
Backward stochastic differential equation, nonlinear expectation and their applications,
in Proceedings of the International Congress of Mathematicians 2010, Vol. I,  393--432.

\bibitem{Sun-Li-Yong 2016}
J. Sun, X. Li, and J. Yong,
Open-loop and closed-loop solvabilities for stochastic linear quadratic optimal control problems,
SIAM J. Control Optim., 54 (2016), 2274--2308.

\bibitem{Sun-Wang2019}
 J. Sun and H. Wang,
 Linear-quadratic optimal control for backward stochastic differential equations with random coefficients,
 ESAIM Control Optim. Calc. Var., 27 (2021), No. 46.

\bibitem{Sun-Wu-Xiong2021}
J. Sun, Z. Wu, and J. Xiong,
Indefinite Backward Stochastic Linear-Quadratic Optimal Control Problems,
arXiv.org/abs/2104.04747v1.

\bibitem{Sun-Yong 2014}
J. Sun and J. Yong,
 Linear quadratic stochastic differential games: Open-loop and closed-loop saddle points,
 SIAM J. Control Optim.,
 52 (2014), 4082--4121.
%
\bibitem{Sun-Yong 2020}
J. Sun and J. Yong,
Stochastic Linear-Quadratic Optimal Control Theory: Open-Loop and Closed-Loop Solutions,
SpringerBriefs in Mathematics, Springer, Cham, 2020.

\bibitem{Wang-Wu-Xiong2012}
G. Wang, Z. Wu, and J. Xiong,
Partial information LQ optimal control of backward stochastic differential equations,
in Proceedings of the 10th World Congress on Intelligent Control and Automation,
IEEE, (2012),  1694--1697.

\bibitem{Wang-Xiao-Xiong2018}
G. Wang, H. Xiao, and J. Xiong,
A kind of LQ non-zero sum differential game of backward stochastic differential equation
with asymmetric information,
Automatica, 97 (2018),  346--352.

\bibitem{Wonham1968}
W. Wonham,
On a matrix Riccati equation of stochastic control,
SIAM J Control, 6 (1968) 312--326.

%

\bibitem{Yong-Zhou1999}
J. Yong and X. Y. Zhou, Stochastic Controls: Hamiltonian Systems and HJB Equations,
Springer-Verlag, New York, 1999.

\bibitem{Zhang2017}
J. Zhang,
Backward Stochastic Differential Equations: From Linear to Fully Nonlinear Theory,
Springer-Verlag, New York, 2017.

\end{thebibliography}
\end{document}